\documentclass[12pt,a4paper, reqno]{amsart}
\usepackage[a4paper]{geometry}
\usepackage{amsthm}
\usepackage{amssymb}
\usepackage{amsmath}
\usepackage{mathtools}
\usepackage[utf8]{inputenc}

\usepackage[T1]{fontenc}

\usepackage{hyperref}

\usepackage{wasysym}
\usepackage{color}
\usepackage{enumitem}

\usepackage{tikz}
\usetikzlibrary{positioning}
\usetikzlibrary{decorations.pathreplacing}
\usetikzlibrary{patterns}
\pgfdeclarepatternformonly[\LineSpace]{my north east lines}{\pgfqpoint{-1pt}{-1pt}}{\pgfqpoint{\LineSpace}{\LineSpace}}{\pgfqpoint{\LineSpace}{\LineSpace}}%
{
    \pgfsetlinewidth{0.4pt}
    \pgfpathmoveto{\pgfqpoint{0pt}{0pt}}
    \pgfpathlineto{\pgfqpoint{\LineSpace + 0.1pt}{\LineSpace + 0.1pt}}
    \pgfusepath{stroke}
}

\newdimen\LineSpace
\tikzset{
    line space/.code={\LineSpace=#1},
    line space=3pt
}

\newtheorem{thm}{Theorem}[section]
\newtheorem{cor}[thm]{Corollary}
\newtheorem{lem}[thm]{Lemma}
\newtheorem{prop}[thm]{Proposition}

\newtheorem{prob}{\sc Problem}
\theoremstyle{definition}

\newtheorem{eks}[thm]{\sc Example}
\theoremstyle{remark}

\numberwithin{equation}{section}

\DeclareMathOperator{\Lip}{Lip}
\DeclareMathOperator{\dent}{dent}

\DeclareMathOperator{\clconv}{\overline{conv}}
\DeclareMathOperator{\conv}{conv}

\newcommand{\seg}[3]{[#1,#2]_{#3}}

\title[Daugavet- and $\Delta$-points in Lipschitz-free spaces]{Characterizations of Daugavet- and delta-points in Lipschitz-free spaces}
\author{Triinu Veeorg}
\address{Institute of Mathematics and Statistics, University of Tartu, Narva~mnt 18, 51009, Tartu, Estonia}
\email{triinu.veeorg@ut.ee}
\thanks{This work was supported by the Estonian Research
Council grant (PRG1598).}

\begin{document}
\begin{abstract}
    A norm one element $x$ of a Banach space is a Daugavet-point (respectively,~a $\Delta$-point) if every slice of the unit ball (respectively,~every slice  of the unit ball containing $x$) contains an element, which is almost at distance 2 from $x$. We characterize Daugavet- and $\Delta$-points in Lipschitz-free spaces. Furthermore, we construct a Lipschitz-free space with the Radon--Nikod\'ym property and a Daugavet-point; this is the first known example of such a Banach space.
\end{abstract}

\maketitle

\section{Introduction}
Let $X$ be a Banach space and $x\in S_X$. According to \cite{AHLP} we say that
\begin{enumerate}
    \item $x$ is a \emph{Daugavet-point} if for every slice $S$ of $B_X$ and for every $\varepsilon>0$ there 
exists $y\in S$ such that $\|x-y\|\ge 2-\varepsilon$;
    \item $x$ is a \emph{$\Delta$-point} if for every slice $S$ of $B_X$ with $x\in S$ and for every $\varepsilon>0$ there 
exists $y\in S$ such that $\|x-y\|\ge 2-\varepsilon$.
\end{enumerate}
These recently introduced concepts opened a new direction in the study of the well-known Daugavet property and diameter-2 properties (see, e.g., \cite{ALM}, \cite{DJR}, \cite{HPV}, \cite{JRZ}). In fact, a Banach space has the Daugavet property (respectively, diametral local diameter two property) if and only if all unit sphere elements are Daugavet-points (respectively, $\Delta$-points) (see \cite{AHLP}, \cite{DW01}).

In this paper we shall explore Daugavet- and $\Delta$-points further in Lipschitz-free spaces, which can be considered a continuation of the study initiated by Jung and Rueda Zoca in \cite{JRZ}.  Throughout the paper, $M$ is a metric space with metric $d$ and a fixed point 0. We denote by $\Lip_0(M)$ the Banach space of all Lipschitz functions $f\colon M\rightarrow\mathbb{R}$ with $f(0)=0$ equipped with the obvious linear structure and the norm
$$\|f\|:=\sup\Big\{\frac{|f(x)-f(y)|}{d(x,y)}\colon x, y\in M, x\neq y\Big\}.$$ 
Let $\delta\colon M\rightarrow \Lip_0(M)^*$ be the canonical isometric embedding of $M$ into $\Lip_0(M)^*$, which is given by $x\mapsto \delta_x$, where $\delta_x(f)=f(x)$. 
The norm closed linear span of $\delta(M)$ in $\Lip_0(M)^*$ is called the \emph{Lipschitz-free space over $M$} and is denoted by $\mathcal{F}(M)$ (see \cite{GOD} and \cite{Weaver} for the background).  An element in $\mathcal{F}(M)$ of the form 
$$m_{xy}:=\frac{\delta_x-\delta_y}{d(x,y)}$$
for $x, y\in M$ with $x\neq y$ is called a \emph{molecule}. We denote the set of all molecules by $\mathcal{M}(M)$.
Clearly $\mathcal{M}(M)\subseteq S_{\mathcal{F}(M)}$ and it is well known that
$$\clconv\big(\mathcal{M}(M)\big)=B_{\mathcal{F}(M)}$$
and
$$\mathcal{F}(M)^*= \Lip_0(M).$$

In \cite{JRZ}, both Daugavet- and $\Delta$-points were studied in Lipschitz-free spaces, and the following characterizations were provided:
\begin{itemize}
    \item If $M$ is a compact metric space, then $\mu\in S_{\mathcal{F}(M)}$ is a Daugavet-point if and only if $\|\mu-\nu\|=2$ for every  $\nu \in \dent (B_{\mathcal{F}(M)})$ (see {\cite[Theorem~3.2]{JRZ}}).
    \item Let $x, y\in M$ with $x\neq y$. Then $m_{xy}\in S_{\mathcal{F}(M)}$ is a $\Delta$-point if and only if for every $\varepsilon>0$ and slice $S$ of $S_{\mathcal{F}(M)}$ with $m_{xy}\in S$ there exist $u, v\in M$ with $u\neq v$ such that $m_{uv}\in S$ and $d(u,v)<\varepsilon$ (see {\cite[Theorem~4.7]{JRZ}}).
\end{itemize}

This left open two important questions:
\begin{itemize}
    \item How to characterize Daugavet-points in Lipschitz-free spaces over any metric space $M$?
    \item How to characterize all $\Delta$-points  in Lipschitz-free spaces?
\end{itemize}

This paper addresses these problems. In Section 2, we provide a generalization of {\cite[Theorem~3.2]{JRZ}} for all metric spaces $M$ (see Theorem \ref{general_daugavet}).

In Section 3, we give an example of a metric space $M$ such that the Lipschitz-free space $\mathcal{F}(M)$ has the Radon--Nikod\'ym property and also a Daugavet-point (see Example \ref{RNP_Daug_ex}). To our knowledge, this is the first example of such a Banach space.

In Section 4, we focus on $\Delta$-points. We generalize {\cite[Theorem~4.7]{JRZ}} for convex combinations of molecules  (see Theorem \ref{Delta_general}). Furthermore, we note that a convex combination of molecules, which are $\Delta$-points, is also a $\Delta$-point, provided it has norm 1  (see Corollary \ref{Delta_cor1}). However, the converse does not hold in general, a convex combination of molecules can be a $\Delta$-point even if none of those molecules is a $\Delta$-point  (see Example \ref{example_delta}).

We consider only real Banach spaces and use standard notation. For a Banach space $X$ we will denote the closed unit ball by $B_X$, the unit sphere by $S_X$ and the dual space by $X^*$. By $B(x,r)$ we denote the closed ball of radius $r$ with the center $x$. For convenience we agree that $B(x,0)=\{x\}$.

First, we observe that {\cite[Theorem~2.6]{JRZ}} can be equivalently written in the following form.
\begin{lem}[cf. {\cite[Theorem~2.6]{JRZ}}]\label{molekul_dis}
    Let $\mu \in S_{\mathcal{F}(M)}$. Then for every $\varepsilon>0$ there exists $\delta>0$ such that if $u, v\in M$ satisfy $0<d(u,v)<\delta$, then $\|\mu - m_{uv}\|\ge 2-\varepsilon$.
\end{lem}

We also present a preliminary result that will be used throughout the paper.
\begin{lem}\label{lip_norm}
Let $x, y,u, v\in M$ with $x\neq y,u\neq v$, and $\varepsilon>0$. The following statements are equivalent:
\begin{enumerate}[label={(\roman*)}]
    \item $\|m_{xy}+m_{uv}\|\ge 2-\varepsilon$;
    \item $d(x,v)+d(u,y)\ge d(x,y)+d(u,v)-\varepsilon\max\big\{d(x,y),d(u,v)\big\}$.
\end{enumerate}
Furthermore, the equalities in $(i)$ and $(ii)$ hold simultaneously and in that case
$$\|m_{xy}+m_{uv}\|= \frac{d(x,v)+d(u,y)+|d(x,y)-d(u,v)|}{\max\big\{d(x,y),d(u,v)\big\}}.$$
\end{lem}

\begin{proof}
Without loss of generality, let us assume that $d(x,y)\ge d(u,v)$.

$(i)\Rightarrow(ii)$. Assume that $\|m_{xy}+m_{uv}\|\ge 2-\varepsilon$. Then
\begin{align*}
    2-\varepsilon&\le \|m_{xy}+m_{uv}\|\\
    &=\frac{\big\|(\delta_x-\delta_v)d(u,v)+(\delta_u-\delta_y)d(u,v)+(\delta_u-\delta_v)\big(d(x,y)-d(u,v)\big)\big\|}{d(x,y)d(u,v)}
    \\
    &\le \frac{d(x,v)+d(u,y)+d(x,y)-d(u,v)}{d(x,y)}.
\end{align*}
Therefore
$$d(x,v)+d(u,y)\ge (1-\varepsilon)d(x,y)+d(u,v)=d(x,y)+d(u,v)-\varepsilon\max\big\{d(x,y),d(u,v)\big\}.$$

$(ii)\Rightarrow(i)$. Assume that 
$$d(x,v)+d(u,y)\ge d(x,y)+d(u,v)-\varepsilon\max\big\{d(x,y),d(u,v)\big\},$$ 
i.e.,
$$d(v,x)+d(y,u)-d(u,v)\ge (1-\varepsilon)d(x,y).$$
Let us examine the function $f\colon M\rightarrow\mathbb{R}$,
$$f(p)=\min\big\{d(y,p),d(v,p)+d(y,u)-d(u,v)\big\}+a,$$
where $a\in \mathbb{R}$ is such that $f(0)=0$. According to {\cite[Proposition~1.32]{Weaver}}, $f\in \Lip_0(M)$ and $\|f\|\le 1$. Let us note that
\begin{align*}
    f(x)&=\min\big\{d(y,x),d(v,x)+d(y,u)-d(u,v)\big\}+a\ge (1-\varepsilon)d(x,y)+a,\\
    f(y)&=\min\big\{0,d(v,y)+d(y,u)-d(u,v)\big\}+a= a,\\
    f(u)&=\min\big\{d(y,u),d(v,u)+d(y,u)-d(u,v)\big\}+a=d(y,u)+a,\\
    f(v)&=\min\big\{d(y,v),0+d(y,u)-d(u,v)\big\}+a=d(y,u)-d(u,v)+a.
\end{align*}
Therefore
$$\|m_{xy}+m_{uv}\|\ge f(m_{xy})+f(m_{uv})\ge 1-\varepsilon+1=2-\varepsilon.$$
Now when the equivalence of $(i)$ and $(ii)$ is proved, it is clear that the equalities in $(i)$ and $(ii)$ hold simultaneously. Assume that the equalities hold both in $(i)$ and $(ii)$. Then
\begin{align*}
    \|m_{xy}+m_{uv}\|&= 2-\varepsilon\\
    &=2-\frac{d(x,y)+d(u,v)-d(x,v)-d(u,y)}{\max\big\{d(x,y),d(u,v)\big\}}\\
    &= \frac{d(x,v)+d(u,y)+|d(x,y)-d(u,v)|}{\max\big\{d(x,y),d(u,v)\big\}}.
\end{align*}
\end{proof}

\section{Characterizations of Daugavet-points in Lipschitz-free spaces}
It has been shown that a Daugavet-point is at distance two from all denting points of the unit ball (see {\cite[Proposition~3.1]{JRZ}}). Furthermore, in Lipschitz-free spaces over compact metric spaces, every unit sphere element that is at distance two from all denting points of the unit ball is a Daugavet-point (see {\cite[Theorem~3.2]{JRZ}}).  The main purpose of this section is to remove the compactness assumption from {\cite[Theorem~3.2]{JRZ}}, i.e., to generalize the result to all metric spaces $M$. 

Let us first introduce some notation for segments and approximate segments between two points (see also \cite{RRZ}). For every $u,v\in M$ and $\delta>0$ let
$$[u,v]:=\big\{p\in M \colon d(u,p)+d(v,p)=d(u,v)\big\}$$
and
$$\seg{u}{v}{\delta}:=\big\{p\in M \colon d(u,p)+d(v,p)<d(u,v)+\delta\big\}.$$

\begin{thm}\label{general_daugavet}
    Let $\mu\in S_{\mathcal{F}(M)}$. The following statements are equivalent:
    \begin{enumerate}[label={(\roman*)}]
        \item $\mu$ is a Daugavet-point;
        \item for every $\nu \in \dent (B_{\mathcal{F}(M)})$ we have $\|\mu-\nu\|=2$;
        \item if for $u, v \in M$ with $u\neq v$, and $r,s>0$ there exists $\delta>0$ such that  $$\seg{u}{v}{\delta}\subseteq B\big(u,rd(u,v)\big)\cup B\big(v,sd(u,v)\big),$$ then $\|\mu-m_{uv}\|\ge 2-2r-2s.$
        
    \end{enumerate}
\end{thm}

To prove $(ii)\Rightarrow(iii)$ of Theorem \ref{general_daugavet} we introduce three lemmas.

\begin{figure}[!ht]
	\begin{center}
		\begin{tikzpicture}[thick, scale=1.5]
		\draw[dashed] (3,0) ellipse (3cm and 0.5cm);
		\draw[rotate=9, dashed, fill=gray!20] (0.75,-0.05) ellipse (0.65cm and 0.15cm);
		\draw[rotate=-2, dashed, fill=gray!20] (3.5,0.2) ellipse (2.45cm and 0.25cm);
		\node [black] at (0.25,0) {$\bullet$}; 
		\node [black] at (5.75,0) {$\bullet$}; 
		\node [black] at (1.25,0.15) {$\bullet$}; 
		\node [right] at (0.25,0) {$u$};
		\node [left] at (5.75,0) {$v$};
		\node [right] at (1.25,0.15) {$x$}; 
		\end{tikzpicture}
	\end{center}
\caption{Illustration of Lemma \ref{subset_Line}} \label{fig_subset_line}
\end{figure}
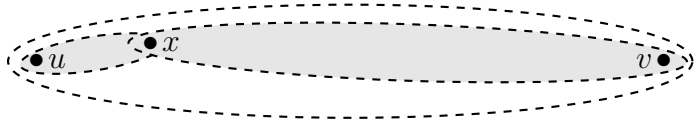

\begin{lem}\label{subset_Line}
     Let $u, v\in M$, $\delta>0$ and $x\in \seg{u}{v}{\delta}$. There exists $\delta'>0$ such that 
     $$\seg{u}{x}{\delta'}\cup\seg{v}{x}{\delta'}\subseteq \seg{u}{v}{\delta}.$$
\end{lem}

\begin{proof}
Since $x\in\seg{u}{v}{\delta}$, there exists $\delta'>0$ such that 
$$d(u,x)+d(v,x)+\delta'<d(u,v)+\delta.$$
If $p\in\seg{u}{x}{\delta'}$, then
\begin{align*}
    d(u,p)+d(v,p)&< d(u,x)+\delta'-d(x,p)+d(v,p)\\
    &\le d(u,x)+\delta'+d(v,x)\\
    &< d(u,v)+\delta,
\end{align*}
giving us $\seg{u}{x}{\delta'}\subseteq \seg{u}{v}{\delta}.$
The inclusion $\seg{v}{x}{\delta'}\subseteq \seg{u}{v}{\delta}$ can be proved analogously. For illustration see Figure \ref{fig_subset_line}.
\end{proof}

\begin{lem}\label{emptyset}
     Let $u, v\in M$ and $r,s,\delta>0$ be such that 
     $$\seg{u}{v}{\delta}\subseteq B(u,r)\cup B(v,s).$$
     For every $\varepsilon>0$ there exist $\delta'>0$, $x\in B(u,r) $ and $y\in B(v,s)$ such that the following holds:
     \begin{enumerate}[label={(\arabic*)}]
         \item $d(u,x)+d(v,y)+d(x,y)<d(u,v)+\delta$;\label{lem_cond4}
         \item $\seg{x}{y}{\delta'}\subseteq\seg{u}{v}{\delta}$;\label{lem_cond2}
         \item $d(x,y)\le d(u,v)$;\label{lem_cond3}
         \item $\seg{x}{y}{\delta'}\subseteq B(x,\varepsilon)\cup B(y,\varepsilon)$.\label{lem_cond1}
     \end{enumerate}
\end{lem}

\begin{figure}[!ht]
	\begin{center}
		\begin{tikzpicture}[thick, scale=1.5]
		\draw[dashed] (3,0) ellipse (3cm and 0.5cm);
		\draw[rotate=-2, dashed] (3.4,0.2) ellipse (2.55cm and 0.3cm);
		\draw[dashed, fill=gray!20] (2.65,0.15) ellipse (1.7cm and 0.16cm);
		\draw (1.25,0) arc (0:45:1cm);
		\draw (1.25,0) arc (0:-45:1cm);
		\draw (1.1,0) arc (0:45:0.85cm);
		\draw (1.1,0) arc (0:-45:0.85cm);
		\draw (4.07,0) arc (180:150:1.68cm);
		\draw (4.07,0) arc (180:210:1.68cm);
		\draw (4.22,0) arc (180:150:1.53cm);
		\draw (4.22,0) arc (180:210:1.53cm);
		\node [black] at (0.25,0) {$\bullet$}; 
		\node [black] at (5.75,0) {$\bullet$}; 
		\node [black] at (1.15,0.15) {$\bullet$}; 
		\node [black] at (4.15,0.15) {$\bullet$}; 
		\node [right] at (0.25,0) {$u$};
		\node [left] at (5.75,0) {$v$};
		\node [below left] at (1.15,0.15) {$x$}; 
		\node [below right] at (4.15,0.15) {$y$}; 
		\node [right] at (1,0.7) {$B(v,r')$}; 
		\node [left] at (4.25,0.7) {$B(u,s')$}; 
		\node [left] at (0.95,0.6) {$B(v,r'-\alpha)$}; 
		\node [right] at (4.3,0.6) {$B(u,s'-\beta)$}; 
		\end{tikzpicture}
	\end{center}
\caption{Illustration of Lemma \ref{emptyset}} \label{fig_line}
\end{figure}
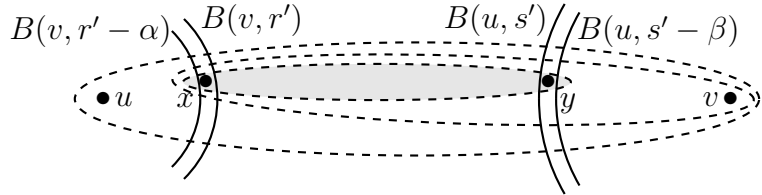

\begin{proof}
Fix $\varepsilon>0$. We may assume that $\delta$ is small enough that $\delta <\varepsilon$. Let $\alpha>0$ be such that $\delta +\alpha<\varepsilon$.
Set
$$r':=\min \big\{t\ge0\colon \seg{u}{v}{\delta}\subseteq B(u,t)\cup B(v,s)\big\}.$$
Clearly $0\le r'\le r$. If $r'\le\varepsilon$, then take $x=u$, otherwise choose
$$x\in\seg{u}{v}{\delta}\setminus \big(B(u,r'-\alpha)\cup B(v,s)\big)$$
(see Figure \ref{fig_line}). Note that $x\in B(u,r')\subseteq B(u,r)$. According to Lemma \ref{subset_Line} there exists $\gamma>0$ such that 
$$\seg{x}{v}{\gamma}\subseteq \seg{u}{v}{\delta}.$$
We may assume that $\gamma$ is small enough to satisfy  $$d(u,x)+d(x,v)+\gamma<d(u,v)+\delta.$$
Then
$$\seg{x}{v}{\gamma}\subseteq \seg{u}{v}{\delta}\subseteq B(u,r')\cup B(v,s).$$
Set
$$s'=\min \big\{t\ge0\colon \seg{x}{v}{\gamma}\subseteq B(u,r')\cup B(v,t)\big\}.$$
Clearly $0\le s'\le s$. If $s'\le\varepsilon$, then take $y=v$, otherwise choose
$$y\in \seg{x}{v}{\gamma}\setminus \big(B(u,r')\cup B(v,s'-\alpha)\big)$$
(see Figure \ref{fig_line}). Note that $y\in B(v,s')\subseteq B(v,s)$. The condition \ref{lem_cond4} holds since
$$d(u,x)+d(v,y)+d(x,y)<d(u,x)+d(x,v)+\gamma<d(u,v)+\delta.$$
According to Lemma \ref{subset_Line} there exists $\delta'>0$ such that
$$\seg{x}{y}{\delta'}\subseteq \seg{x}{v}{\gamma}.$$
Additionally we may assume that $\delta'$ is small enough to satisfy 
$$d(u,x)+d(v,y)+d(x,y)+\delta'<d(u,v)+\delta.$$
Then \ref{lem_cond2} holds since $\seg{x}{v}{\gamma}\subseteq \seg{u}{v}{\delta}$.
If $r'>\varepsilon$, then 
$$d(u,x)>r'-\alpha>\varepsilon-\alpha>\delta $$
and therefore 
$$d(x,y)< d(u,v)+\delta-d(u,x)-d(v,y)< d(u,v).$$
Analogously, $d(x,y)<d(u,v)$, if $s'>\varepsilon$. On the other hand, if $r'\le\varepsilon$ and $s'\le\varepsilon$, then $x=u$ and $y=v$, hence $d(x,y)\le d(u,v)$. Thus we have \ref{lem_cond3}. 

In order to show \ref{lem_cond1}, let $p\in \seg{x}{y}{\delta'}$. Then
$$p\in \seg{x}{y}{\delta'}\subseteq \seg{x}{v}{\gamma}\subseteq B(u,r')\cup B(v,s').$$
Assume that $p\in B(u,r')$ (the case $p\in B(v,s')$ is analogous). If $r'\le\varepsilon$, then $x=u$ and $p\in B(u,r')\subseteq B(x,\varepsilon)$. Otherwise $d(u,x)>r'-\alpha$ and then
\begin{align*}
    d(x,p)&< d(x,y)+\delta'-d(y,p)\\
    &<d(u,v)+\delta-d(u,x)-d(v,y)-d(y,p)\\
    &\le \delta-d(u,x)+d(u,p)\\
    &\le\delta-d(u,x)+r'\\
    &<\delta +\alpha\\
    &< \varepsilon.
\end{align*}
This gives us $p\in B(x,\varepsilon)\cup B(y,\varepsilon)$ and therefore \ref{lem_cond1} holds.
\end{proof}

\begin{lem}\label{lemma_dent}
Let $M$ be complete, and let $u, v\in M$ and $r,s,\delta>0$ with $r+s<d(u,v)$ be such that
$$\seg{u}{v}{\delta}\subseteq B(u,r)\cup B(v,s).$$
Then there exist $x\in B(u,r)$ and $y\in B(v,s)$ such that $m_{xy}$ is a denting point.
\end{lem}

\begin{proof}
We shall find suitable $x$ and $y$ as limits of two convergent sequences $(x_n)$ and $(y_n)$ of elements in the sets 
$$B(u,r)\cap \seg{u}{v}{\delta}\quad\text{
and}\quad
B(v,s)\cap \seg{u}{v}{\delta},$$
respectively. We will construct these sequences inductively and in the process we also define two null sequences $(\delta_n)$ and $(\varepsilon_n)$ of positive numbers. Let $\varepsilon_1>0$ be such that 
$r+s+2\varepsilon_1<d(u,v)$,
and by applying Lemma \ref{emptyset} (for $u=u$, $v=v$, $r=r$, $s=s$, $\delta=\delta$ and $\varepsilon=\varepsilon_1$) we obtain $x_1$ ($=x$), $y_1$ ($=y$) and $\delta_1$ ($=\delta'$). We may additionally assume that $\delta_{1}<\varepsilon_1$.

Suppose that we have found $x_n$, $y_n$, $\delta_n$ and $\varepsilon_n$ for $n\in \mathbb{N}$. 
Let $\varepsilon_{n+1}\in(0,\delta_n/6)$. By applying Lemma \ref{emptyset} (for $u=x_n$, $v=y_n$, $r=\varepsilon_n$, $s=\varepsilon_n$, $\delta=\varepsilon_{n+1}$ and $\varepsilon=\varepsilon_{n+1}$) we obtain $x_{n+1}$ ($=x$), $y_{n+1}$  ($=y$) and $\delta_{n+1}$ ($=\delta'$). We may additionally assume that $\delta_{n+1}<\varepsilon_{n+1}$.

Clearly for every $n\in\mathbb{N}$ we have
\begin{enumerate}[label={(\arabic*)}]
    \item $d(x_n,x_{n+1})+d(y_n,y_{n+1})+d(x_{n+1},y_{n+1})<d(x_n,y_n)+\varepsilon_{n+1}$;\label{dent_cond4}
    \item $\seg{x_n}{y_n}{\delta_n}\subseteq\seg{u}{v}{\delta}$;\label{dent_cond2}
    \item $d(x_{n+1},y_{n+1})\le d(x_{n},y_{n})$;\label{dent_cond3}
    \item $\seg{x_n}{y_n}{\delta_n}\subseteq B(x_n,\varepsilon_n)\cup B(y_n,\varepsilon_n)$.\label{dent_cond1}
\end{enumerate}
Furthermore, $\varepsilon_{n+1}<\delta_{n}/6<\varepsilon_{n}/6$ for every $n\in\mathbb{N}$. Then  for every $m,n\in\mathbb{N}$ with $m>n$ we have 
$$d(x_n,x_m)\le\sum_{i=n}^{m-1}d(x_i,x_{i+1})\le\sum_{i=n}^{m-1}\varepsilon_i<2\varepsilon_n\rightarrow0$$
and therefore the sequence $(x_n)$ converges to some element  $x\in M$. Analogously $(y_n)$ converges to some element  $y\in M$.
For every $n\in \mathbb{N}$ we get
$$d(v,x_n)\ge d(u,v)-d(u,x_1)-d(x_1,x_n) > d(u,v)-r-2\varepsilon_1>s,$$
thus $x_n\notin B(v,s)$. 
Therefore, from
$$x_n\in\seg{x_n}{y_n}{\delta_n}\subseteq\seg{u}{v}{\delta}\subseteq B(u,r)\cup B(v,s)$$
we get $x_n\in B(u,r)$, which yields $x\in B(u,r)$.
Analogously 
$y\in B(v,s).$
Also $x\neq y$, since $r+s<d(u,v)$.

By {\cite[Theorem~2.6]{GPPR}} we see that $m_{xy}$ is a denting point if and only if for every $\varepsilon>0$ there exists $\gamma>0$ such that
$$\seg{x}{y}{\gamma}\subseteq B(x,\varepsilon)\cup B(y,\varepsilon).$$
Fix $\varepsilon>0$. Let $n\in\mathbb{N}$ be such that $d(x_n,x)<\varepsilon/2$,  $d(y_n,y)<\varepsilon/2$ and $\varepsilon_n<\varepsilon/2$. By \ref{dent_cond4} and \ref{dent_cond3} we have
\begin{align*}
    d(x_n,x_m)+d(y_n,y_m)
    &\le d(x_n,x_{n+1})+d(y_n,y_{n+1})+d(x_{n+1},x_{m})+d(y_{n+1},y_{m}) \\
    &<d(x_{n},y_{n})+\varepsilon_{n+1}-d(x_{n+1},y_{n+1})+2\varepsilon_{n+1}+2\varepsilon_{n+1}\\
    &\le d(x_n,y_n)+5\varepsilon_{n+1}-d(x,y)
\end{align*}
for every $m>n$, which gives 
$$d(x_n,x)+d(y_n,y)\le d(x_n,y_n)+5\varepsilon_{n+1}-d(x,y).$$
If $p\in \seg{x}{y}{\varepsilon_{n+1}}$, then
\begin{align*}
    d(x_n,p)+d(y_n,p)&\le d(x_n,x)+d(y_n,y)
    +d(x,p)+d(y,p)\\
    &<d(x_n,y_n)+5\varepsilon_{n+1}-d(x,y)+d(x,y)+\varepsilon_{n+1}\\
    &< d(x_n,y_n)+\delta_n.
\end{align*}
Therefore
$$\seg{x}{y}{\varepsilon_{n+1}}\subseteq \seg{x_n}{y_n}{\delta_n}.$$
Furthermore, $B(x_n,\varepsilon_n)\subseteq B(x,\varepsilon)$ since 
$$d(x_n,x)+\varepsilon_n<\frac{\varepsilon}{2}+\frac{\varepsilon}{2}=\varepsilon.$$
Analogously $B(y_n,\varepsilon_n)\subseteq B(y,\varepsilon)$ and then 
$$\seg{x}{y}{\varepsilon_{n+1}}\subseteq \seg{x_n}{y_n}{\delta_n}\subseteq B(x_n,\varepsilon_n)\cup B(y_n,\varepsilon_n)\subseteq B(x,\varepsilon)\cup B(y,\varepsilon).$$
According to {\cite[Theorem~2.6]{GPPR}}, $m_{xy}$ is a denting point.
\end{proof}

Now we are ready to prove Theorem \ref{general_daugavet}.

\begin{proof}[Proof of Theorem \ref{general_daugavet}.]
$(i)\Rightarrow(ii)$ is {\cite[Proposition~3.1]{JRZ}}.

$(ii)\Rightarrow(iii).$ Assume that $\|\mu-\nu\|=2$ for every $\nu \in \dent (B_{\mathcal{F}(M)})$. First we will prove $(ii)\Rightarrow(iii)$ providing that $M$ is complete.

Note that the case $r+s\ge1$ is trivial. Fix $u, v\in M$ with $u\neq v$ and $r,s>0$ with $r+s<1$, such that there exists $\delta>0$ with $$\seg{u}{v}{\delta}\subseteq B\big(u,rd(u,v)\big)\cup B\big(v,sd(u,v)\big).$$
According to Lemma \ref{lemma_dent} there exist $x\in B\big(u,rd(u,v)\big)$ and $y\in B\big(v,sd(u,v)\big)$ such that $m_{xy}$ is a denting point. By $(ii)$ we have $\|\mu-m_{xy}\|=2$. We know that
$$d(u,x)+d(v,y)\le rd(u,v)+sd(u,v)<d(u,v).$$
Then equalities must hold for some $\varepsilon>0$ in Lemma \ref{lip_norm} and we get
%
\begin{align*}
    \|m_{xy}-m_{uv}\|&=\frac{d(u,x)+d(v,y)+|d(u,v)-d(x,y)|}{\max\big\{d(u,v),d(x,y)\big\}}\\
    &\le \frac{2d(u,x)+2d(v,y)}{\max\big\{d(u,v),d(x,y)\big\}}\\
    &\le \frac{2rd(u,v)+2sd(u,v)}{\max\big\{d(u,v),d(x,y)\big\}}\\
    &\le 2(r+s).
\end{align*}
Consequently,
$$\|\mu-m_{uv}\|\ge \|\mu-m_{xy}\|-\|m_{xy}-m_{uv}\|\ge 2-2r-2s.$$
This closes the case when $M$ is complete.

Now we assume that $M$ is any metric space and let $M'$ be its completion. Then $\mathcal{F}(M)=\mathcal{F}(M')$ and therefore for every $\nu \in \dent (B_{\mathcal{F}(M')})$ we have $\|\mu-\nu\|=2$.

Let us note that, if for $u, v\in M$ with $u\neq v$ and $s,r>0$ there exists $\delta>0$ such that  in $M$ we have
$$\seg{u}{v}{\delta}\subseteq B\big(u,rd(u,v)\big)\cup B\big(v,sd(u,v)\big),$$
then in $M'$ we have the same inclusion
and by the first case we get $\|\mu-m_{uv}\|\ge 2-2r-2s$.

$(iii)\Rightarrow(i)$. Assume that $(iii)$ holds. We will show that $\mu$ is a Daugavet-point. Fix $\varepsilon>0$ and a slice $S(f,\alpha)$ of $ B_{\mathcal{F}(M)}$. We will prove that there exist $u, v\in M$ with $u\neq v$ such that $m_{uv}\in S(f,\alpha)$ and $\|\mu-m_{uv}\|\ge 2-\varepsilon$.

Let $u_0\neq v_0\in M$ be such that $f(u_0)-f(v_0)>(1-\alpha)d(u_0,v_0)$. According to Lemma \ref{molekul_dis} there exists $\gamma>0$ such that if $x\neq y\in M$ satisfy $d(x,y)<\gamma$, then $\|\mu-m_{xy}\|\ge 2-\varepsilon$. Let $n\in \mathbb{N}$ and $\delta>0$ be such that 
$$\Big(1-\frac{\varepsilon}{4}+\delta\Big)^n d(u_0,v_0)<\gamma$$
and  $f(u_0)-f(v_0)>(1-\alpha)(1+\delta)^{n}d(u_0,v_0)$. 

If $\|\mu-m_{u_0v_0}\|\ge 2-\varepsilon$, then we have found suitable points $u$ and $v$.

Consider the case where $\|\mu-m_{u_0v_0}\|< 2-\varepsilon$.
By $(iii)$ there exists 
$$p\in \seg{u_0}{v_0}{\delta d(u_0,v_0)}\setminus\Big( B\big(u_0,\frac{\varepsilon}{4}d(u_0,v_0)\big)\cup B\big(v_0,\frac{\varepsilon}{4}d(u_0,v_0)\big)\Big).$$
Therefore
\begin{align*}
    f(u_0)-f(p)+f(p)-f(v_0)&>(1-\alpha)(1+\delta)^{n}d(u_0,v_0)\\
    &>(1-\alpha)(1+\delta)^{n-1}\big(d(u_0,p)+d(v_0,p)\big).
\end{align*}
Then either $$f(u_0)-f(p)>(1-\alpha)(1+\delta)^{n-1}d(u_0,p)$$
or 
$$f(p)-f(v_0)>(1-\alpha)(1+\delta)^{n-1}d(v_0,p).$$
Additionally we have
$$d(u_0,p)<(1+\delta)d(u_0,v_0)-d(v_0,p)< \Big(1-\frac{\varepsilon}{4}+\delta\Big)d(u_0,v_0).$$
Analogously  $d(v_0,p)< \big(1-\varepsilon/4+\delta\big)d(u_0,v_0).$
Therefore there exist $u_1, v_1\in M$ such that
$$f(u_1)-f(v_1)>(1-\alpha)(1+\delta)^{n-1}d(u_1,v_1)$$
and $d(u_1,v_1)<(1-\varepsilon/4+\delta)d(u_0,v_0)$. 

Now we will repeat this step as many times as needed, but no more than $n$ times. Assume for  $k\in \{1,\ldots,n-1\}$ that
$$f(u_k)-f(v_k)>(1-\alpha)(1+\delta)^{n-k}d(u_k,v_k)$$
and $d(u_k,v_k)<(1-\varepsilon/4+\delta)^kd(u_0,v_0)$. 

If $\|\mu-m_{u_kv_k}\|\ge 2-\varepsilon$,
then we may choose $u_k$ and $v_k$ as points $u$ and $v$ we were looking for. 

If $\|\mu-m_{u_kv_k}\|< 2-\varepsilon$,
then as we did before, we can find $u_{k+1}, v_{k+1}\in M$ such that
$$f(u_{k+1})-f(v_{k+1})>(1-\alpha)(1+\delta)^{n-k-1}d(u_{k+1},v_{k+1})$$
and $d(u_{k+1},v_{k+1})<(1-\varepsilon/4+\delta)^{k+1}d(u_0,v_0)$.

If by the n-th step we have not found suitable points, then we have
$f(u_n)-f(v_n)>(1-\alpha)d(u_n,v_n)$ and 
$$d(u_n,v_n)<\Big(1-\frac{\varepsilon}{4}+\delta\Big)^n d(u_0,v_0)<\gamma,$$
which gives us $\|\mu-m_{u_nv_n}\|\ge 2-\varepsilon$. Now we have found suitable points $u$ and $v$,  therefore $\mu$ is a Daugavet-point.
\end{proof}

In the case, where $\mu$ is a molecule we can use  Lemma \ref{lip_norm} to simplify condition $(iii)$ of Theorem \ref{general_daugavet}.
\begin{cor}\label{mol_daugavet}
    Let $x, y\in M$ be such that $x\neq y$. The following statements are equivalent:
    \begin{enumerate}[label={(\roman*)}]
        \item $m_{xy}$ is a Daugavet-point;
        \item If for $u, v \in M$ with $u\neq v$, and $r,s>0$ there exists $\delta>0$ such that  $$\seg{u}{v}{\delta}\subseteq B\big(u,rd(u,v)\big)\cup B\big(v,sd(u,v)\big),$$  then $$d(x,u)+d(y,v)\ge d(x,y)+d(u,v)-2(r+s)\max\big\{d(x,y),d(u,v)\big\}.$$
        
    \end{enumerate}
\end{cor}

\section{Lipschitz-free space with Radon--Nikod\'ym property and with Daugavet-point}
In this section, we construct a Lipschitz-free space with the Radon--Nikod\'ym property and a Daugavet-point. Recall that a Banach space $X$ has the \emph{Radon--Nikod\'ym property} if every nonempty bounded closed convex set is the closed convex hull of its denting points. The Radon--Nikod\'ym property and the Daugavet property can be considered the opposite -- a Banach space with the Radon--Nikod\'ym property has slices of the unit ball with arbitrary small diameter, the Daugavet property implies that all slices of the unit ball have diameter 2. Therefore it was somewhat surprising to find such an example.

Let us notice that if a Banach space $X$ has the Radon--Nikod\'ym property, then $x\in S_X$ is a Daugavet-point if and only if all denting points of the unit ball are at distance 2 from $x$. Sufficiency is already proved by {\cite[Proposition~3.1]{JRZ}}. Also, if $X$ has the Radon--Nikod\'ym property, then $B_X=\clconv\big(\dent (B_X)\big)$, i.e., every slice of the unit ball contains a denting point of the unit ball  and therefore if all denting points of the unit ball are at distance 2 from $x$, then clearly $x$ is a Daugavet-point. This implies that Lipschitz-free spaces are a good candidate for finding said example.

Before presenting the example, we recall that for Lipschitz-free spaces the Radon--Nikod\'ym property is equivalent to the Schur property (see {\cite[Theorem~4.6]{AGPP}}). We also remind that the Lipschitz-free space over a countable complete metric space has the Schur property (see {\cite[Corollary~2.7]{ANPP}}). Therefore, our aim is to find a  countable complete metric space $M$ such that $\mathcal{F}(M)$ has a Daugavet-point.

\begin{eks}\label{RNP_Daug_ex}
Let $x:=(0,0)$, $y:=(1,0)$ and $S_0:=\{x,y\}$. For every $n\in \mathbb{N}$ let
$$S_{n}:=\Big\{\Big(\frac{k}{2^n},\frac{1}{2^{n}}\Big)\colon k\in\{0,1,\ldots,2^n\}\Big\}$$
(see Figure \ref{fig_RNP}).
Consider
$$M:=\bigcup_{n=0}^\infty S_n$$
with the metric
$$d\big((a_1,b_1),(a_2,b_2)\big):=\begin{cases} 
    |a_1-a_2|, & \text{if }b_1=b_2,\\ 
    \min\{a_1+a_2,2-a_1-a_2\}+|b_1-b_2|, &\text{if } b_1\neq b_2. \end{cases}$$
    
\begin{figure}[ht!]
	\begin{center}
		\begin{tikzpicture}
		\draw[step=1cm,gray,very thin] (-0.1,-0.1) grid (8.1,4.1);
		\draw[very thin] (0,4) -- (8,4);
		\draw[very thin] (0,2) -- (8,2);
		\draw[very thin] (0,1) -- (8,1);
		\draw[very thin] (0,0.5) -- (8,0.5);
		\draw[very thin] (0,0) -- (0,4);
		\draw[very thin] (8,0) -- (8,4);
		\node [below] at (0,0) {$0$};
		\node [below] at (2,0) {$\frac{1}{4}$};
		\node [below] at (4,0) {$\frac{1}{2}$};
		\node [below] at (6,0) {$\frac{3}{4}$};
		\node [below] at (8,0) {$1$};
		\node [left] at (0,2) {$\frac{1}{4}$};
		\node [left] at (0,4) {$\frac{1}{2}$};
		\node [left] at (0,0) {$x$};
		\node [right] at (8,0) {$y$};
		\node [black] at (0,0) {$\bullet$};
		\node [black] at (8,0) {$\bullet$};
		\node [black] at (0,4) {$\bullet$};
		\node [black] at (4,4) {$\bullet$};
		\node [black] at (8,4) {$\bullet$};
		\node [black] at (0,2) {$\bullet$};
		\node [black] at (2,2) {$\bullet$};
		\node [black] at (4,2) {$\bullet$};
		\node [black] at (6,2) {$\bullet$};
		\node [black] at (8,2) {$\bullet$};
		\node [black] at (0,1) {$\bullet$};
		\node [black] at (1,1) {$\bullet$};
		\node [black] at (2,1) {$\bullet$};
		\node [black] at (3,1) {$\bullet$};
		\node [black] at (4,1) {$\bullet$};
		\node [black] at (5,1) {$\bullet$};
		\node [black] at (6,1) {$\bullet$};
		\node [black] at (7,1) {$\bullet$};
		\node [black] at (8,1) {$\bullet$};
		\node [black] at (0,0.5) {$\bullet$};
		\node [black] at (0.5,0.5) {$\bullet$};
		\node [black] at (1,0.5) {$\bullet$};
		\node [black] at (1.5,0.5) {$\bullet$};
		\node [black] at (2,0.5) {$\bullet$};
		\node [black] at (2.5,0.5) {$\bullet$};
		\node [black] at (3,0.5) {$\bullet$};
		\node [black] at (3.5,0.5) {$\bullet$};
		\node [black] at (4,0.5) {$\bullet$};
		\node [black] at (4.5,0.5) {$\bullet$};
		\node [black] at (5,0.5) {$\bullet$};
		\node [black] at (5.5,0.5) {$\bullet$};
		\node [black] at (6,0.5) {$\bullet$};
		\node [black] at (6.5,0.5) {$\bullet$};
		\node [black] at (7,0.5) {$\bullet$};
		\node [black] at (7.5,0.5) {$\bullet$};
		\node [black] at (8,0.5) {$\bullet$};
		\end{tikzpicture}
	\end{center}
\caption{The sets $S_0,\ldots,S_4$} \label{fig_RNP}
\end{figure}
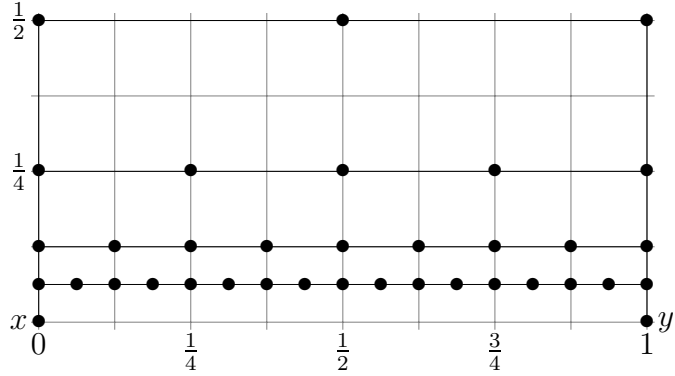

First we show that $M$ is complete. Let $(u_n)$ be a Cauchy sequence in $M$. To show that $(u_n)$ converges to an element of $M$ we consider two cases. 

\begin{enumerate}
    \item First assume that there exists $m\in\mathbb{N}$ such that $(u_n)$ is in $\cup_{n=0}^m S_n$. The set $\cup_{n=0}^m S_n$ is finite and therefore $(u_n)$ is eventually constant.
    \item Assume that for every $m\in\mathbb{N}$, there exist $k>m$ and $n\in\mathbb{N}$ such that $u_n\in S_k$. Choose a subsequence $(u_{n_k})$ such that $u_{n_k}\in S_{m_k}$, where $m_1<m_2<m_3<\cdots$. By definition, the distance between any two different elements $u_{n_k}=(a_{n_k},b_{n_k})$ and $u_{n_l}=(a_{n_l},b_{n_l})$ is 
    $$\min\{a_{n_k}+a_{n_l},2-a_{n_k}-a_{n_l}\}+|b_{n_k}-b_{n_l}|.$$ 
    Since $(u_{n_k})$ is a Cauchy sequence, either $u_{n_k}\rightarrow x$ or $u_{n_k}\rightarrow y$.
\end{enumerate}
According to {\cite[Corollary~2.7]{ANPP}}, $\mathcal{F}(M)$ has the Schur property, since $M$ is countable and complete and by  {\cite[Theorem~4.6]{AGPP}}, $\mathcal{F}(M)$ also has the Radon--Nikod\'ym property.

We will now show that $m_{xy}$ is a Daugavet-point. 
By Theorem \ref{general_daugavet}, it suffices to show that $\|m_{xy}-\nu\|=2$ for every $\nu \in \dent (B_{\mathcal{F}(M)})$. Note that all denting points are preserved extreme points and according to {\cite[Corollary~3.44]{Weaver}} $\nu$ is a preserved extreme point in $B_{\mathcal{F}(M)}$ only if $\nu=m_{uv}$ for some $u, v\in M$ with $u\neq v$. Furthermore, by {\cite[Theorem~2.6]{GPPR}}, if $m_{uv}$ is a denting point of $B_{\mathcal{F}(M)}$, then for every $\varepsilon>0$ there exists $\delta>0$ such that 
$$\seg{u}{v}{\delta}\subseteq B(u,\varepsilon)\cup B(v,\varepsilon).$$
Therefore neither $m_{xy}$ nor $m_{yx}$ is a denting point of $B_{\mathcal{F}(M)}$ because for every $n\in\mathbb{N}$ we have
$$z:=\big(1/2,1/2^{n+2}\big)\in \seg{x}{y}{1/2^{n}}\setminus \big(B(x,1/2)\cup B(y,1/2)\big),$$
since 
$d(x,z)=d(y,z)=1/2+1/2^{n+2}.$

Now fix $m_{uv} \in \dent (B_{\mathcal{F}(M)})$ with $u=(a_1,b_1)$ and $v=(a_2,b_2)$.
Clearly $[u,v]=\{u,v\}$. Let us show that $\|m_{xy}-m_{uv}\|=2$. 

If $b_1=b_2$, then $|a_1-a_2|=b_1$, because otherwise either $(a_1+b_1,b_1)$ or $(a_1-b_1,b_1)$ is in $[u,v]\setminus\{u,v\}$.
Hence 
$$d(x,u)+d(y,v)=a_1+b_1+1-a_2+b_2\ge 1+|a_1-a_2|=d(x,y)+d(u,v)$$
and by Lemma \ref{lip_norm} we get $\|m_{xy}-m_{uv}\|=2$. 

If $b_1\neq b_2$, then either $a_1=a_2=0$ or $a_1=a_2=1$, because otherwise one of the four points $(0,b_1),(0,b_2),(1,b_1),(1,b_2)$ is in $[u,v]\setminus\{u,v\}$. Hence 
$$d(x,u)+d(y,v)=a_1+b_1+1-a_2+b_2=1+b_1+b_2\ge 1+|b_1-b_2|=d(x,y)+d(u,v)$$
and by Lemma \ref{lip_norm} we get $\|m_{xy}-m_{uv}\|=2$. 

Now we have shown that for every $\nu \in \dent (B_{\mathcal{F}(M)})$ we have $\|m_{xy}-\nu\|=2$ and therefore $m_{xy}$ is a Daugavet-point.


\end{eks}

Our Example \ref{RNP_Daug_ex} and the fact that a Banach space with the Daugavet property cannot have the Radon--Nikod\'ym property inspire the following question.
\begin{prob}
How large does the set of Daugavet-points in Banach space 
need to be in order to ensure that the Banach space fails the Radon--Nikod\'ym property?
\end{prob}

\section{Characterization of Delta-points in Lipschitz-free spaces}
Every Daugavet-point is clearly a $\Delta$-point. In Lipschitz-free spaces a $\Delta$-point in  not necessarily a Daugavet-point (see {\cite[Example~4.4]{JRZ}}).
A molecule $m_{xy}$ is a $\Delta$-point if and only if for every $\varepsilon>0$ and a slice $S$ with $m_{xy}\in S$ there exist $u, v\in M$ with $u\neq v$ such that $m_{uv}\in S$ and $d(u,v)<\varepsilon$ (see {\cite[Theorem~4.7]{JRZ}}). Moreover, the only if part can be proved similarly for any $\mu\in S_{\mathcal{F}(M)}$.

\begin{prop}[cf. {\cite[Theorem~4.7]{JRZ}}]\label{Delta_nec}
Let $\mu\in S_{\mathcal{F}(M)}$ be such that for every $\varepsilon>0$ and a slice $S$ of $S_{\mathcal{F}(M)}$ with $\mu\in S$ there exist $u, v\in M$ with $u\neq v$ such that $m_{uv}\in S$ and $d(u,v)<\varepsilon$. Then $\mu$ is a $\Delta$-point.
\end{prop}

We do not know whether the converse of Proposition \ref{Delta_nec} holds in general. In this section we will prove this for the case where $\mu$ is a convex combination of molecules.

For $\mu\in S_{\mathcal{F}(M)}$ let
$$\mathcal{M}(\mu):=\{m_{uv}\colon \mu=\lambda m_{uv}+(1-\lambda)\nu \text{ for some }\lambda\in(0,1]\text{ and } \nu\in S_{\mathcal{F}(M)}\}.$$
Note that $f\in S_{\Lip_0(M)}$ with $f(\mu)=1$ satisfies $f(m_{uv})=1$ for every $m_{uv}\in\mathcal{M}(\mu)$.

Let $n\in \mathbb{N}$, $\lambda_1,\ldots,\lambda_n>0$ with $\sum^{n}_{i=1}\lambda_i=1$, and $m_{x_1 y_1},\ldots,m_{x_n y_n}\in S_{\mathcal{F}(M)}$ be such that $\mu:=\sum^{n}_{i=1}\lambda_i m_{x_iy_i}\in S_{\mathcal{F}(M)}$. According to {\cite[Theorem~2.4]{RRZ}}, for every sequence $k_1,\ldots,k_{m+1}\in \{1,\ldots,n\}$ with $k_1=k_{m+1}$, we have
\begin{equation}\label{distance_eq}
    \sum^{m}_{j=1} d(x_{k_j},y_{k_{j+1}})\ge\sum^{m}_{j=1} d(x_{k_j},y_{k_j}).
\end{equation}
Furthermore, equality in \eqref{distance_eq} yields $m_{x_{k_j}y_{k_{j+1}}}\in \mathcal{M}(\mu)$ for every $j\in\{1,\ldots,m\}$ with $x_{k_j}\neq y_{k_{j+1}}$. Indeed, suppose that
$$\sum^{m}_{j=1} d(x_{k_j},y_{k_{j+1}})=\sum^{m}_{j=1} d(x_{k_j},y_{k_j}).$$
For simplicity we will also assume that $k_1,\ldots,k_{m}$ are pairwise distinct.
Set
$$\lambda_0:=\min_{i\in\{1,\ldots,n\}}\frac{\lambda_i}{d(x_i,y_i)}$$
and
$$l_i:=\begin{cases} 
    \lambda_i-\lambda_0d(x_i,y_i), & \text{if }i\in\{k_1,\ldots,k_m\},\\ 
    \lambda_i, &\text{if } i\in\{1,\ldots,n\}\setminus\{k_1,\ldots,k_m\}. \end{cases}$$ 
Clearly $l_i\ge 0$ for every $i\in \{1,\ldots,n\}$. Then
\begin{align*}
    \mu&=\sum^{n}_{i=1}l_im_{x_iy_i}+\lambda_0\sum^{m}_{j=1}d(x_{k_j},y_{k_{j}})m_{x_{k_j}y_{k_j}}\\
    &=\sum^{n}_{i=1}l_im_{x_iy_i}+\lambda_0\sum^{m}_{\substack{j=1\\\hskip2.8mm x_{k_j}\neq y_{k_{j+1}}}}d(x_{k_j},y_{k_{j+1}})m_{x_{k_j}y_{k_{j+1}}}
\end{align*}
and
$$\sum^{n}_{i=1}l_i+\lambda_0\sum^{m}_{j=1}d(x_{k_j},y_{k_{j+1}})=\sum^{n}_{i=1}l_i+\lambda_0\sum^{m}_{j=1}d(x_{k_j},y_{k_{j}})=1.$$

In the proof of {\cite[Theorem~4.7]{JRZ}}, the function
$$f_{xy}(p)=\frac{d(x,y)}{2}\cdot\frac{d(y,p)-d(x,p)}{d(x,p)+d(y,p)},$$
where $x, y\in M$ with $x\neq y$, played a key role. To generalize  {\cite[Theorem~4.7]{JRZ}} we first need to find a function $f_\mu$ with similar properties as $f_{xy}$ for any $\mu\in\conv\big(\mathcal{M}(M)\big)\cap S_{\mathcal{F}(M)}$, which is accomplished by the following lemmas.

\begin{lem}\label{function}
    Let $n\in \mathbb{N}$, $\lambda_1,\ldots,\lambda_n>0$ with $\sum^{n}_{i=1}\lambda_i=1$, and $m_{x_1 y_1},\ldots,m_{x_n y_n}\in S_{\mathcal{F}(M)}$ be such that  $\mu:=\sum^{n}_{i=1}\lambda_i m_{x_iy_i}\in S_{\mathcal{F}(M)}$. There exists $f\in S_{\Lip_0(M)}$ with
    $f(\mu)=1$
    such that for all $i,j\in\{1,\ldots,n\}$ with $x_{i}\neq y_{j}$ the following conditions are equivalent:
    \begin{enumerate}[label={(\roman*)}]
        \item $f(m_{x_{i}y_{j}})=1$;
        \item $m_{x_{i}y_{j}}\in\mathcal{M}(\mu)$.
    \end{enumerate}
\end{lem}
\begin{proof}
As we noticed earlier, $(ii)\Rightarrow(i)$ for any $f\in S_{\Lip_0(M)}$ with $f(\mu)=1$. Therefore it is sufficient to find $f\in S_{\Lip_0(M)}$ with $f(\mu)=1$ such that $(i)\Rightarrow(ii)$ holds as well. We will start with any function $g\in S_{\Lip_0(M)}$ such that $g(\mu) =1$.
Let 
$$A:=\big\{(i,j)\colon x_{i}\neq y_{j}, m_{x_{i}y_{j}}\notin\mathcal{M}(\mu)\big\}.$$
If $A$ is empty, then we can take $f=g$. From here on we will assume that $A\neq\emptyset$. For every $k:=(k_1,k_2)\in A$ we will define a Lipschitz function $h_k$ such that $h_k(x_{k_1})-h_k(y_{k_2})<d(x_{k_1},y_{k_2})$ and $h_k(x_{i})-h_k(y_{i})=d(x_{i},y_{i})$ for every $i\in\{1,\ldots,n\}$. 

Fix $k:=(k_1,k_2)\in A$. 
If $g(x_{k_1})-g(y_{k_2})<d(x_{k_1},y_{k_2})$, then let $h_{k}=g$.

Now we consider the case where $g(x_{k_1})-g(y_{k_2})=d(x_{k_1},y_{k_2})$. Set
$$M_0:=\{x_1,\ldots,x_n,y_1,\ldots,y_n\}\subseteq M.$$
Let us first define the function $h_{k}$ on $M_0$. To do so, we will first define a set of indices. Let $B\subseteq\{1,\ldots,n\}$ be such that $i\in B$ if and only if there exist $m\in \mathbb{N}$ and $k_3,\ldots,k_{m+1}\in \{1,\ldots,n\}$ with $k_{m+1}=i$ such that 
$$g(x_{k_{j}})-g(y_{k_{j+1}})=d(x_{k_{j}},y_{k_{j+1}})$$
for every $j\in\{1,\ldots,m\}$.
Note that $k_2\in B$, because from $g(\mu)=1$ we get $g(x_{k_2})-g(y_{k_2})=d(x_{k_2},y_{k_2})$.

Suppose that $k_1\in B$. Then there exist $m\in \mathbb{N}$ and $k_3,\ldots,k_{m+1}\in \{1,\ldots,n\}$ with $k_{m+1}=k_1$ such that 
$$g(x_{k_{j}})-g(y_{k_{j+1}})=d(x_{k_{j}},y_{k_{j+1}})$$
for every $j\in\{1,\ldots,m\}$. 
By reshuffling we get
$$\sum^{m}_{j=1} d(x_{k_j},y_{k_{j+1}})=\sum^{m}_{j=1}\big(g(x_{k_j})-g(y_{k_{j+1}})\big)=\sum^{m}_{j=1}\big(g(x_{k_j})-g(y_{k_j})\big)=\sum^{m}_{j=1} d(x_{k_j},y_{k_j}).$$
Hence $m_{x_{k_1}y_{k_{2}}}\in\mathcal{M}(\mu)$, which contradicts $(k_1,k_2)\in A$. Therefore $k_1\notin B$.

Let $C:=\{x_i\colon i\in B\}\cup\{y_i\colon i\in B\}$. Suppose that $g(p)-g(q)=d(p,q)$ for some $p\in C$ and $q\in M_0$. We will show that $q\in C$. Let $i\in B$ and $j\in \{1,\ldots,n\}$ be such that $p\in \{x_{i},y_{i}\}$ and $q\in \{x_{j},y_{j}\}$. Then
\begin{align*}
    g(x_{i})-g(y_{j})&=g(x_{i})-g(p)+g(p)-g(q)+g(q)-g(y_{j})\\
    &=d(x_{i},p)+d(p,q)+d(q,y_{j})\\
    &\ge d(x_{i},y_{j}).
\end{align*}
Therefore $j\in B$, i.e., $q\in C$. From this we deduce that if $p\in C$ and $q\in M_0\setminus C$, then $g(p)-g(q)<d(p,q)$.

Let $\delta>0$ be such that for every $p,q\in M_0$, if $g(p)-g(q)<d(p,q)$, then $g(p)-g(p)+\delta<d(p,q)$. Finally we are ready to define the function $h_{k}$ on $M_0$. Let $h_k\colon M_0\rightarrow \mathbb{R}$ be such that
$$h_{k}(p)=\begin{cases} 
    g(p)+\delta, & \text{if } p\in C,\\ 
    g(p), & \text{if }  p\in M_0\setminus C.
    \end{cases}$$
If $p\in C$ and $q\in M_0\setminus C$, then $g(p)-g(q)+\delta<d(p,q)$, which gives us
$$|h_{k}(p)-h_{k}(q)|=|g(p)-g(q)+\delta|<d(p,q).$$
Therefore the Lipschitz constant of $h_k$ on $M_0$ is smaller than or equal to 1. Furthermore, for every $i\in \{1,\ldots,n\}$ either $x_i,y_i\in C$ or $x_i,y_i\notin C$ giving us 
$$h_{k}(x_i)-h_{k}(y_i)=g(x_i)-g(y_i)=d(x_i,y_i).$$
Last we point out that
$$h_{k}(x_{k_1})-h_{k}(y_{k_2})=g(x_{k_1})-g(y_{k_2})-\delta<d(x_{k_1},y_{k_2}).$$ 
We will extend $h_{k}$ from $M_0$ to $M$ by the 
McShane--Whitney Theorem. Then the Lipschitz constant of $h_k$ on $M$ is 1.

Now we have defined $h_{k}$ for every $k\in A$. Let
$$f:=\frac{1}{|A|}\sum_{k\in A}h_{k}+a,$$
where $a\in \mathbb{R}$ is such that $f(0)=0$. Clearly $\|f\|\le 1$ and 
$$f(\mu)=\frac{1}{|A|}\sum_{k\in A}\sum^{n}_{i=1}\lambda_i \frac{h_{k}(x_i)-h_{k}(y_i)}{d(x_i,y_i)}=1.$$
Furthermore, for every $k=(k_1,k_2)\in A$ we have
\begin{align*}
    f(x_{k_1})-f(y_{k_2})&=\frac{1}{|A|}\sum_{l\in A}\big(h_{l}(x_{k_1})-h_{l}(y_{k_2})\big)\\
    &\le \frac{|A|-1}{|A|}d(x_{k_1},y_{k_2})+\frac{1}{|A|}\big(h_{k}(x_{k_1})-h_{k}(y_{k_2})\big)\\
    &<d(x_{k_1},y_{k_2}).
\end{align*}
With this we conclude the proof.
\end{proof}

\begin{lem}\label{f_mu}
    Let $n\in \mathbb{N}$, $\lambda_1,\ldots,\lambda_n>0$ with $\sum^{n}_{i=1}\lambda_i=1$, and $m_{x_1 y_1},\ldots,m_{x_n y_n}\in S_{\mathcal{F}(M)}$ be such that  $\mu:=\sum^{n}_{i=1}\lambda_i m_{x_iy_i}\in S_{\mathcal{F}(M)}$. There exist $f_\mu\in S_{\Lip_0(M)}$ and $\delta>0$ such that the following holds:
    \begin{enumerate}[label={(\arabic*)}]
        \item $f_\mu(\mu) =1$;
        \item For every $u,v\in M$ and $\alpha\in(0,\delta)$ with $u\neq v$ and $m_{uv}\in S(f_{\mu},\alpha)$ there exist $i,j\in \{1,\ldots,n\}$ with $x_{i}\neq y_{j}$ such that $m_{x_iy_j}\in\mathcal{M}(\mu)$
        and
        $$(1 - \alpha)\max\{d(x_{i},v) + d(y_{j},v),d(x_{i},u) + d(y_{j},u)\}< d(x_{i},y_{j}).$$
    \end{enumerate}
\end{lem}
\begin{proof}
By Lemma \ref{function} there exists $g\in S_{\Lip_0(M)}$ with $g(\mu)=1$
such that for $i, j\in \{1,\ldots,n\}$  with $x_{i}\neq y_{j}$ we have $g(m_{x_iy_j})=1$ if and only if $m_{x_iy_j}\in\mathcal{M}(\mu)$.
For every $i\in \{1,\ldots,n\}$ let $h_i\colon M\rightarrow \mathbb{R}$ be such that
$$h_i(p)=\max\Big\{\frac{g(x_i)-g(y_j)}{d(x_i,p)+d(y_j,p)}d(x_i,p)\colon j\in\{1,\ldots,n\}, x_{i}\neq y_{j}\Big\}.$$
Clearly $g(x_i)-g(y_i)>0$, therefore $h_i(p)\ge0$ for every $i\in\{1,\ldots,n\}$ and $p\in M$. Define $f_\mu\colon M\rightarrow \mathbb{R}$ by
$$f_\mu(p)=\max_{i\in\{1,\ldots,n\}}\big\{g(x_i)-h_i(p)\big\}+a,$$
where $a\in \mathbb{R}$ is such that $f_{\mu}(0)=0$.
Note that for all $i, j\in \{1,\ldots,n\}$ we have $d(x_i,y_j)\ge g(x_i)-g(y_j)$ and
\begin{align}
    \nonumber\frac{d(x_i,y_j)}{2}-f_{x_iy_j}(p)&=\frac{d(x_i,y_j)}{2}\frac{d(x_i,p)+d(y_j,p)-\big(d(y_j,p)-d(x_i,p)\big)}{d(x_i,p)+d(y_j,p)}\\
    &=\frac{d(x_i,y_j)}{g(x_i)-g(y_j)}\frac{g(x_i)-g(y_j)}{d(x_i,p)+d(y_j,p)}d(x_i,p),\label{eq_delta}
\end{align}
then from {\cite[Proposition~1.32]{Weaver}}  and {\cite[Lemma~3.6]{GPR}} we get $\|f_\mu\|\le 1$. For every $i\in \{1,\ldots,n\}$ we have 
$$f_\mu(x_i)\ge g(x_i)-h_i(x_i)+a= g(x_i)+a.$$ 
For fixed $j\in \{1,\ldots,n\}$ let $i\in \{1,\ldots,n\}$ be such that  $f_\mu(y_j)=g(x_i)-h_i(y_j)+a$. If $x_{i}= y_{j}$, then $f_\mu(y_j)=g(y_j)-h_i(y_j)+a\le g(y_j)+a$,
otherwise 
$$f_\mu(y_j)=g(x_i)-h_i(y_j)+a\le g(x_i)-\frac{g(x_i)-g(y_j)}{d(x_i,y_j)+d(y_j,y_j)}d(x_i,y_j)+a=g(y_j)+a.$$
This gives us $f_\mu(x_i)-f_\mu(y_j)\ge g(x_i)-g(y_j)$ for every $i,j\in \{1,\ldots,n\}$ and therefore $f_\mu(\mu)\ge g(\mu)=1$. Note that $\|f_\mu\|\le1$, hence $f_\mu\in S_{\Lip_0(M)}$ and $f_\mu(\mu)=1$.

Choose $\delta>0$ such that if $g(x_i)-g(y_j)<d(x_i,y_j)$ for some $i,j\in\{1,\ldots,n\}$, then $g(x_i)-g(y_j)<(1-\delta)d(x_i,y_j)$.  We will now show that condition $(2)$ holds. Fix $u, v\in M$ with $u\neq v$ and $\alpha\in(0,\delta)$ such that $m_{uv}\in S(f_{\mu},\alpha)$. 
Let $i\in \{1,\ldots,n\}$ be such that $f_\mu(u)=g(x_{i})-h_{i}(u)+a.$
Also $f_\mu(v)\ge g(x_{i})-h_{i}(v)+a$ and therefore
$$(1-\alpha)d(u,v)<f_\mu(u)-f_\mu(v)\le h_{i}(v)-h_{i}(u).$$
There exists $j\in \{1,\ldots,n\}$ with $x_i\neq y_j$ such that
$$h_{i}(v)=\frac{g(x_{i})-g(y_{j})}{d(x_{i},v)+d(y_{j},v)}d(x_{i},v).$$ 
By \eqref{eq_delta} we get that
\begin{align*}
    (1-\alpha)d(u,v)&< h_{i}(v)-h_{i}(u)\\
    &\le \frac{g(x_{i})-g(y_{j})}{d(x_{i},v)+d(y_{j},v)}d(x_{i},v)-\frac{g(x_{i})-g(y_{j})}{d(x_{i},u)+d(y_{j},u)}d(x_{i},u)\\
    &=\frac{g(x_{i})-g(y_{j})}{d(x_{i},y_{j})}\big(f_{x_{i}y_{j}}(u)-f_{x_{i}y_{j}}(v)\big)\\
    &\le\min\Big\{\frac{g(x_{i})-g(y_{j})}{d(x_{i},y_{j})}d(u,v),f_{x_{i}y_{j}}(u)-f_{x_{i}y_{j}}(v)\Big\}.
\end{align*}
From {\cite[Lemma~3.6]{GPR}} we get $$d(x_{i},y_{j})>(1-\alpha)\max\big\{d(x_{i},u)+d(y_{j},u),d(x_{i},v)+d(y_{j},v)\big\}.$$ 
Furthermore, 
$$g(x_{i})-g(y_{j})>(1-\alpha)d(x_{i},y_{j})>(1-\delta)d(x_{i},y_{j})$$
and therefore $g(x_{i})-g(y_{j})=d(x_{i},y_{j})$, i.e., $m_{x_{i}y_{j}}\in\mathcal{M}(\mu)$.
With this we conclude the proof.
\end{proof}

\begin{thm}\label{Delta_general}
Let $\mu\in\conv\big(\mathcal{M}(M)\big)\cap S_{\mathcal{F}(M)}$. Then $\mu$ is a $\Delta$-point if and only if for every $\varepsilon>0$ and a slice $S$ of $S_{\mathcal{F}(M)}$ with $\mu\in S$ there exist $u, v\in M$ with $u\neq v$ such that $m_{uv}\in S$ and $d(u,v)<\varepsilon$.
\end{thm}

\begin{proof}
Necessity is already proved in Proposition \ref{Delta_nec}.

Assume that $\mu$ is a $\Delta$-point. Let $n\in\mathbb{N}$, $\lambda_1,\ldots,\lambda_n>0$ with $\sum^{n}_{i=1}\lambda_i=1$, and $m_{x_1 y_1},\ldots,m_{x_n y_n}\in S_{\mathcal{F}(M)}$ be such that  $\mu=\sum^{n}_{i=1}\lambda_im_{x_iy_i}$.

According to Lemma \ref{f_mu} there exist $f_\mu\in S_{\Lip_0(M)}$ and $\delta>0$ such that
$f_\mu(\mu) =1$
and for every $u, v\in M$ with $u\neq v$ and $\delta'\in(0,\delta)$ with $m_{uv}\in S(f_{\mu},\alpha)$ there exist $i,j\in \{1,\ldots,n\}$ with $x_i\neq y_j$ such that $m_{x_iy_j}\in\mathcal{M}(\mu)$
        and
        $$(1 - \alpha)\max\{d(x_{i},v) + d(y_{j},v),d(x_{i},u) + d(y_{j},u)\}< d(x_{i},y_{j}).$$
For every $i,j\in \{1,\ldots,n\}$ with $m_{x_iy_j}\in\mathcal{M}(\mu)$ let $l_{ij}\in(0,1]$ and $\nu_{ij}\in S_{\mathcal{F}(M)}$ be such that $\mu=l_{ij}m_{x_iy_j}+(1-l_{ij})\nu_{ij}$.

Fix $\varepsilon>0$ and a slice $S=S(f,\alpha)$ of $S_{\mathcal{F}(M)}$ such that $\mu\in S$. According to {\cite[Lemma~2.1]{IK}} we can assume that $\alpha<\delta$ and
$$\Big(\frac{1}{(1-\alpha)^2}-1\Big)\max_{i,j\in\{1,\ldots,n\}}d(x_i,y_j)<\varepsilon.$$
Our aim is to show that there exist $u, v\in M$ with $u\neq v$ such that $m_{uv}\in S$ and $d(u,v)<\varepsilon$.
Set $g=f+f_\mu$.
It is easy to see that $g(\mu)=f_\mu(\mu)+f(\mu)>2-\alpha$, i.e., $\mu\in S\big(g/\|g\|,1-(2-\alpha)/\|g\|\big)$. 
Since $\mu$ is a $\Delta$-point, by {\cite[Remark~2.4]{JRZ}} there exist $u, v\in M$ with $u\neq v$ such that $g(m_{uv})>2-\alpha$
and
$$\|\mu-m_{uv}\|\ge 2-\alpha\min\big\{l_{ij}\colon i,j\in\{1,\ldots,n\}, m_{x_iy_j}\in\mathcal{M}(\mu)\big\}.$$
It is easy to see that $f_\mu(m_{uv})>1-\alpha$ and
$f(m_{uv})>1-\alpha,$
i.e., $m_{uv}\in S(f,\alpha)$. Now we will show that $d(u,v)<\varepsilon$.

Since $f_\mu(m_{uv})>1-\alpha$ and $\alpha\in (0,\delta)$, there exist $i,j\in \{1,\ldots,n\}$ with $x_i\neq y_j$ such that $m_{x_iy_j}\in\mathcal{M}(\mu)$
        and
\begin{equation}\label{thm4eq1}
    (1 - \alpha)\max\{d(x_{i},v) + d(y_{j},v),d(x_{i},u) + d(y_{j},u)\}< d(x_{i},y_{j}).
\end{equation}

Then $\|\mu-m_{uv}\|\ge 2-\alpha l_{ij}$ and from $\mu=l_{ij} m_{x_{i}y_{j}}+(1-l_{ij})\nu_{ij}$ we get
\begin{align*}
    2-\alpha l_{ij}&\le \|\mu-m_{uv}\|\\
    &\le l_{ij} \|m_{x_{i}y_{j}}-m_{uv}\|+(1-l_{ij})\|\nu-m_{uv}\|\\
    &\le l_{ij}\|m_{x_{i}y_{j}}-m_{uv}\|+2-2l_{ij},
\end{align*}
i.e., $\|m_{x_{i}y_{j}}-m_{uv}\|\ge 2-\alpha$. From $m_{x_iy_j}\in\mathcal{M}(\mu)$ we get $f_\mu(m_{x_{i}y_{j}})=1$ and then
$$\|m_{x_{i}y_{j}}+m_{uv}\|\ge f_\mu(m_{x_{i}y_{j}})+f_\mu(m_{uv}) > 2-\alpha.$$
By Lemma \ref{lip_norm} we get that 
\begin{align}\label{thm4eq2}
    \nonumber&\min\big\{d(x_{i},v)+d(y_{j},u),d(x_{i},u)+d(y_{j},v)\big\}\\
    &\qquad \ge d(x_{i},y_{j})+d(u,v)-\alpha\max\big\{d(x_{i},y_{j}),d(u,v)\big\}\\
    \nonumber&\qquad> (1-\alpha)\big(d(x_{i},y_{j})+d(u,v)\big).
\end{align}
By \eqref{thm4eq1} and \eqref{thm4eq2} we have
\begin{align*}
    d(u,v)&< \frac{d(x_{i},v)+d(y_{j},u)+d(x_{i},u)+d(y_{j},v)}{2(1-\alpha)}-d(x_{i},y_{j})\\
    &<\frac{2d(x_{i},y_{j})}{2(1-\alpha)^2}-d(x_{i},y_{j})\\
    &\le\Big(\frac{1}{(1-\alpha)^2}-1\Big)\max_{i',j'\in\{1,\ldots,n\}}d(x_{i'},y_{j'})\\
    &<\varepsilon.
\end{align*}
Consequently we have found $u, v\in M$ with $u\neq v$ such that $m_{uv}\in S$ and $d(u,v)<\varepsilon$.
\end{proof}

\begin{cor}\label{Delta_cor1}
    Let $(m_{x_iy_i})_{i\in I}$ be finite or infinite sequence of $\Delta$-points. Then if for some choice of $\lambda_i>0$ with $\sum_{i\in I}\lambda_i=1$ we have $\sum_{i\in I}\lambda_im_{x_iy_i}\in S_{\mathcal{F}(M)}$, then $\sum_{i\in I}\lambda_im_{x_iy_i}$ is a $\Delta$-point.
\end{cor}

\begin{proof}
Let $\lambda_i>0$, $i\in I$, be such that $\sum_{i\in I}\lambda_i=1$ and $\mu:=\sum_{i\in I}\lambda_im_{x_iy_i}\in S_{\mathcal{F}(M)}$. Fix $\varepsilon>0$ and a slice $S$ of $S_{\mathcal{F}(M)}$ such that $\mu\in S$.
There exists $i\in I$ such that $m_{x_iy_i}\in S$. Since $m_{x_iy_i}$ is a $\Delta$-point, by Theorem \ref{Delta_general} there exist $u, v \in M$ with $u\neq v$ such that $m_{uv}\in S$ and $d(u,v)<\varepsilon$. According to Proposition \ref{Delta_nec} that means $\mu$ is a $\Delta$-point.
\end{proof}

It is natural to ask whether the converse of this corollary holds. The following example shows that in general it does not.

\begin{eks}\label{example_delta}
Let $M=\big\{(a,b)\colon a\in \{0,1\},b\in [0,1]\big\}\subseteq(\mathbb{R}^2,\|\cdot\|_\infty)$ and
choose $x_1=(0,0)$, $y_1=(1,0)$, $x_2=(1,1)$ and $y_2=(0,1)$ (see Figure \ref{fig_delta_exs}). 

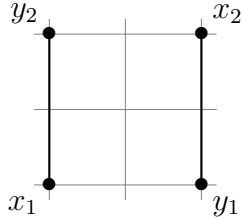
\begin{figure}[ht!]
	\begin{center}
		\begin{tikzpicture}
		\draw[step=1cm,gray,very thin] (-0.2,-0.2) grid (2.2,2.2);
		\draw[thick,black] (0,0) -- (0,2);
		\draw[thick,black] (2,2) -- (2,0);
		\node [below left] at (0,0) {$x_1$};
		\node [below right] at (2,0) {$y_1$};
		\node [above left] at (0,2) {$y_2$};
		\node [above right] at (2,2) {$x_2$};
		
		\node [black] at (0,0) {$\bullet$};
		\node [black] at (2,0) {$\bullet$};
		\node [black] at (0,2) {$\bullet$};
		\node [black] at (2,2) {$\bullet$};
		\end{tikzpicture}
	\end{center}
\caption{Metric space $M$ from Example \ref{example_delta}} \label{fig_delta_exs}
\end{figure}

We have $d(x_1,y_1)=d(x_1,y_2)=d(x_2,y_1)=d(x_2,y_2)=1$ and from Lemma \ref{lip_norm} we derive
$$\|m_{x_1y_1}+m_{x_2y_2}\|=\|m_{x_1y_2}+m_{x_2y_1}\|=2.$$
By {\cite[Corollary~4.9]{JRZ}},
$m_{x_1y_1}$ and $m_{x_2y_2}$ are not $\Delta$-points. However, from {\cite[Proposition~4.2]{JRZ}} we see that $m_{x_1y_2}$ and $m_{x_2y_1}$ are $\Delta$-points. Therefore
$$\frac{1}{2}m_{x_1y_1}+\frac{1}{2}m_{x_2y_2}=\frac{1}{2}m_{x_1y_2}+\frac{1}{2}m_{x_2y_1}$$
is a $\Delta$-point according to Corollary \ref{Delta_cor1}.
\end{eks}

By Theorem \ref{Delta_general}, we have provided a characterization for $\Delta$-points in Lipschitz-free spaces among convex combinations of molecules. However, we are not sure whether the same characterization holds for all unit sphere elements.
\begin{prob}
Let $\mu\in S_{\mathcal{F}(M)}$ be a $\Delta$-point. Does every slice $S$ of $S_{\mathcal{F}(M)}$ with $\mu\in S$ contain such a molecule $m_{uv}$ that the distance $d(u,v)$ is arbitrarily small?
\end{prob}

Example \ref{example_delta} showed us that a convex combination of molecules  can be a $\Delta$-point even if none of those molecules is a $\Delta$-point. However, the point introduced in the example can be presented as convex combination of molecules which are all $\Delta$-points, leaving us with the following question.
\begin{prob}
For a $\Delta$-point $\mu\in S_{\mathcal{F}(M)}\cap \mathcal{M}(M)$, does there exist $\lambda_{1},\ldots,\lambda_{n}>0$ with $\sum^{n}_{i=1}\lambda_i=1$, and $\Delta$-points $m_{x_1y_1},\ldots,m_{x_ny_n}$ such that $\mu=\sum^{n}_{i=1}\lambda_im_{x_iy_i}$?
\end{prob}

\section*{Acknowledgements}
This paper is a part of the author's Ph.D.\ thesis, which is being prepared at University of Tartu under the supervision of Rainis Haller and Vegard Lima. 
The author is grateful to her supervisors for their valuable help and guidance. 
The author is also grateful to Mingu Jung and Abraham Rueda Zoca for sharing their preprint and for drawing the author's attention to fact that Corollary \ref{Delta_cor1} also holds for infinite convex combinations of molecules.


\bibliographystyle{amsplain}

\end{document}